\documentclass{article}

\usepackage[utf8]{inputenc}
\usepackage{amsmath,amsthm,amssymb,amsfonts,amstext, mathtools,thmtools,thm-restate,pinlabel}

\usepackage[T1]{fontenc}
\usepackage[stretch=10]{microtype}
\usepackage{mathtools}
\usepackage[hidelinks, linktoc=all]{hyperref}
\usepackage[capitalize,noabbrev]{cleveref}
\usepackage{fullpage}
\usepackage{lipsum}

\newcommand\MTkillspecial[1]{
  \bgroup
  \catcode`\&=9
  \let\\\relax%
  \scantokens{#1}%
  \egroup
}
\newcommand\DeclarePairedDelimiterMultiline[3]{
  \DeclarePairedDelimiter{#1}{#2}{#3}
  \reDeclarePairedDelimiterInnerWrapper{#1}{star}{
  \mathopen{##1\vphantom{\MTkillspecial{##2}}\kern-\nulldelimiterspace\right.}
  ##2
  \mathclose{\left.\kern-\nulldelimiterspace\vphantom{\MTkillspecial{##2}}##3}}
}

\title{Random walk speed is a proper function on Teichmüller space}

\author{Aitor Azemar\thanks{ Aitor.Azemar@glasgow.ac.uk, University of Glasgow, UK}, Vaibhav Gadre\thanks{Vaibhav.Gadre@glasgow.ac.uk,University of Glasgow, UK}, S\'ebastien Gou\"ezel\thanks{sebastien.gouezel@univ-rennes1.fr, IRMAR, CNRS UMR 6625, Université de Rennes 1, 35042 Rennes, France},\\Thomas Haettel\thanks{thomas.haettel@umontpellier.fr, IMAG, Univ Montpellier, CNRS, France, and IRL 3457, CRM-CNRS, Universit\'{e} de Montr\'{e}al, Canada.}, Pablo Lessa\thanks{plessa@fing.edu.uy, IMERL, Universidad de la República, Uruguay}\,\, and Caglar Uyanik\thanks{caglar@math.wisc.edu, University of Wisconsin, Madison, USA}}

\newtheorem{maintheorem}{Theorem}

\newtheorem{theorem}{Theorem}[section]
\newtheorem{lemma}[theorem]{Lemma}
\newtheorem{corollary}[theorem]{Corollary}
\newtheorem{proposition}[theorem]{Proposition}
\newtheorem{conjecture}[theorem]{Conjecture}
\newtheorem*{conjecture_unnumbered}{Conjecture}
\newtheorem{question}[theorem]{Question}
\newtheorem{definition}[theorem]{Definition}
\DeclareMathOperator{\dist}{dist}

\newcommand{\R}{\mathbb{R}}
\newcommand{\N}{\mathbb{N}}

\newcommand{\T}{\mathcal{T}}
\newcommand{\C}{\mathbb{C}}
\newcommand{\Pbb}{\mathbb{P}}
\newcommand{\Ebb}{\mathbb{E}}
\newcommand{\dd}{\mathop{}\!\mathrm{d}}
\DeclarePairedDelimiterMultiline{\abs}{\lvert}{\rvert}
\DeclarePairedDelimiterMultiline{\pare}{(}{)}
\DeclarePairedDelimiterMultiline{\norm}{\lVert}{\rVert}
\newcommand{\st}{\,:\,}

\renewcommand{\H}{\mathbb{H}}
\renewcommand{\P}{\mathbb{P}}
\renewcommand{\epsilon}{\varepsilon}

\DeclareMathOperator{\isom}{Isom}

\DeclareMathOperator{\Sl}{SL}
\DeclareMathOperator{\supp}{supp}
\DeclareMathOperator{\length}{length}
\DeclareMathOperator{\ext}{Ext}

\numberwithin{equation}{section}

\begin{document}
 \maketitle

\begin{abstract}
Consider a closed surface $M$ with negative Euler characteristic, and an admissible probability measure on the fundamental group of $M$ with finite first moment.  Corresponding to each point in the Teichmüller space of $M$, there is an associated random walk on the hyperbolic plane. We show that the speed of this random walk is a proper function on the Teichmüller space of $M$, and we relate the growth of the speed to the Teichmüller distance to a basepoint. 
One key argument is an adaptation of Gou\"ezel's pivoting techniques to actions of a fixed group on a sequence of hyperbolic metric spaces. 
\end{abstract}

\section{Introduction and statements of results}

Let $M$ be a compact oriented surface with negative Euler characteristic with a basepoint \(p\), 
and let $\Gamma=\pi_1(M,p)$. Let $\mu$ be a probability measure on $\Gamma$ that is admissible, i.e., 
the semigroup generated by the support of $\mu$ is equal to $\Gamma$. 
Further assume that $\mu$ has finite first moment. Consider a random walk $Z_n=g_1\dotsm g_n$ 
where $g_i$ are i.i.d.\ elements of $\Gamma$ with distribution $\mu$. Fixing a complete hyperbolic metric $\rho$ on $M$, define 
\[
\ell(\rho) \coloneqq \lim_{n\to\infty}\frac{\abs{Z_n}_{\rho}}{n}
\]
where $\abs{Z_n}_{\rho}$ denotes the $\rho$-length of the unique hyperbolic geodesic representing the free homotopy class of the element $Z_n$. The limit above exists almost surely, and is well defined. The quantity $\ell(\rho)$ is called the \emph{speed (or drift)} of the random walk on $(M, \rho)$. 

Let $\T(M)$ be the Teichmüller space of marked complete hyperbolic metrics on $M$. 
Our first result is concerned with the quantitative behavior of $\ell(\rho)$ as one varies the hyperbolic metric $\rho$ on $M$, in other words as $[\rho]$ varies over points in the Teichmüller space $\T(M)$. 
When one moves in Teichmüller space, some curves get longer but others get shorter, so the 
behavior is not obvious. However, one expects that most curves get longer, so one should expect 
\(\ell(\rho)\) to tend to infinity at \(\rho\) diverges to infinity in Teichmüller space. There is a 
difficulty, though, that most curves become more and more parallel to each other (up to orientation) when \(\rho\) converges to 
a point at infinity, as they align asymptotically with the measured foliation at infinity. This means that 
consecutive steps of the random walk are likely to be both large, but in opposite directions, 
thereby cancelling each other effectively and not contributing to the speed. 
Our main theorem shows that the former effect dominates the latter: The speed indeed tends to 
infinity at infinity. However, this discussion hints at the fact that this is not straightforward, 
and indeed our proof is rather indirect.

\begin{maintheorem}[Speed is a proper function on Teichmüller space]\label{maintheorem}
Assume that the probability measure $\mu$ on $\Gamma$ is admissible and has a
finite first moment. Then the function \(\ell:\T(M) \to [0,+\infty)\) is
proper.
\end{maintheorem}

Recall that a function is proper if the preimage of any compact set is
compact.  In our particular case, since \(\ell\) is non-negative, properness is
equivalent to the statement that for each \(C > 0\) there exists a compact
set \(K \subset \T(M)\) such that \(\ell(\rho) > C\) for all \([\rho] \notin
K\).

We can actually be more precise, by comparing the speed of a random walk to the Teichmüller distance in the Teichmüller space. Let $(\T(M), \dist_{Teich})$ denote the Teichm\"uller space of $M$ endowed with the Teichm\"uller metric (see \Cref{preliminaries} for the definitions). Then, we have: 

\begin{maintheorem}\label{teichdistancebound}
 For each \([\rho_0] \in \T(M)\) there exists a constant \(c > 0\) such that
 \[\ell(\rho) \ge c\dist_{Teich}([\rho],[\rho_0])\]
 for all \([\rho] \in \T(M)\).
\end{maintheorem}

Our strategy to prove the properness of the drift function in Theorem~\ref{maintheorem} and 
the quantitative bound in Theorem~\ref{teichdistancebound} is through a
compactification argument: we get representations at boundary points, and argue that the drift there is nonzero.
Then, by continuity (up to a natural rescaling), we deduce that the drift \emph{inside} the Teichmüller space is positive.
Moreover, using extremal length bounds we prove that the natural rescaling factor grows at least linearly along all Teichmüller rays.

Note that the speed can also be considered as a Lyapunov exponent. More precisely, if $[\rho] \in \T(M)$ is a point 
in Teichmüller space, we can consider $[\rho]$ as a conjugacy class of a discrete, faithful 
representation $\rho:\pi_1(M) \rightarrow \operatorname{PSL}(2,\R)$. 
Indeed, if we fix a matrix norm $\norm{\cdot}$ on $\operatorname{PSL}(2,\R)$, we have
\[\ell(\rho) = \lim\limits_{n \to +\infty} \frac{1}{n} \int \log \norm{\rho(g_1) \dotsm \rho(g_n)} \dd\mu(g_1) \dotsb \dd\mu(g_n),\]
which is the Lyapunov exponent of the random walk on $\operatorname{PSL}(2,\R)$.

A related argument was used in~\cite{dujardin-favre} to study continuity of Lyapunov exponents for certain meromorphic families of representations in \(\Sl(2,\C)\), with the same idea which consists in looking at the scaled limiting action on an $\R$-tree. One simple situation where both approaches can be used is the case where the hyperbolic structure degenerates by only pinching a simple closed curve: this degeneracy can also be described by a meromorphic family of representations. 
However, in the context of representations of surface groups into $\operatorname{PSL}(2,\R)$, the setting of our continuity result below is more general.

The main feature of this argument is that the representations at boundary points do not live on the same space as the original representations: the group acts on
an $\R$-tree instead of the hyperbolic disk. These representations are constructed in~\cite{b1988} and~\cite{p1988}
(following previous work~\cite{cm1987} and~\cite{ms1984}). In particular, the topological type of Gromov boundaries changes in the  limit. This means that the usual continuity argument for the drift, relying on the convergence of
stationary measures on the boundary (see~\cite{erschler_kaim}), does not work.
However, we are able to obtain
the continuity of the drift in this context thanks to the pivotal times argument
of~\cite{g2021}. The notions in the next Theorem are defined in \Cref{sec:converging}.


\begin{maintheorem}[Speed is lower semi-continuous for a converging sequence of actions on uniformly hyperbolic spaces]
\label{thm:converging_actions} Let \(\Gamma\) be a countable group, and \(\mu_k\) be
a sequence of probability measures on \(\Gamma\) converging pointwise to a
probability measure \(\mu_\infty\). Let $(X_k, o_k)_{k\in \N\cup \{\infty\}}$ be
a uniformly hyperbolic sequence of pointed metric spaces, and let $\rho_k:\Gamma
\to \isom(X_k), k = 1,2,\dotsc$ be a sequence of isometric actions of \(\Gamma\)
converging to an action $\rho_\infty$. Assume that $\mu_\infty$ is
non-elementary for $\rho_\infty$.

Consider for each $k$ a random walk \(Z_n^{(k)} = g_1^{(k)}\dotsm g_n^{(k)}\) where \(g_1^{(k)}, \dotsc, g_n^{(k)},\dotsc\) is an i.i.d.\ sequence with common distribution \(\mu_k\). Then, one has
\[\liminf \ell(\mu_k) \geq \ell(\mu_\infty),\]
where for each \(k \in \N \cup \lbrace \infty\rbrace\) we define \(\ell(\mu_k) \in [0,\infty]\) as the almost sure limit
\[\ell(\mu_k) = \lim\limits_{n \to +\infty}\frac{1}{n}\dist(o_k,\rho_k(Z_n^{(k)})o_k).\]
\end{maintheorem}

Returning to the setting of Theorem~\ref{maintheorem}, it is well known that \(\ell(\rho) > 0\) for all \(\rho\) and the random walk driven by $\mu$ converges to the boundary almost surely. In order words, the limit
\[X_{\infty} = \lim\limits_{n \to \infty}\rho(Z_n)o,\]
exists almost surely and $X_{\infty}$ is in the visual boundary \(\partial \H\) of \(\H\), (cf.~\cite{Kai}). This limit defines a \emph{hitting measure} on $\partial\H$ as follows: for any Borel set $U\subset\partial\H$, 
\[
\nu_{\rho}(U) \coloneqq \mathbb{P}(\lim\limits_{n \to \infty}\rho(Z_n)o \in U).
\]

The measure $\nu_{\rho}$ is the unique \(\mu\)-stationary measure on the visual boundary for the $\rho$-action. 
Properties of the stationary measures associated to random walks with finite support are quite subtle, 
as illustrated by the following. Here $\dim(\nu_\rho) \in [0,1]$ denotes the Hausdorff dimension of the 
stationary measure $\nu_\rho$ associated to the representation $\rho$.

\begin{conjecture_unnumbered}[Singularity Conjecture]\label{singularityconjecture}
If \(\mu\) is admissible and has finite support, then there exists \(\kappa < 1\) such that
\begin{equation*}
  \dim(\nu_\rho) \le \kappa
\end{equation*}
for all \([\rho] \in T(M)\).
\end{conjecture_unnumbered}

The conjecture above, stated in~\cite[Conjecture 1.21]{dkn2009} and more generally in~\cite{kl2011}, 
remains open in spite of some recent progress made in~\cite{tk2020}. We remark that for all \(\rho\) there exists \(\mu\) \emph{with infinite support} on \(\Gamma\) such that \(\dim(\nu_\rho) = 1\).  This follows from the Furstenberg-Lyons-Sullivan discretization of Brownian motion~\cite{ls1984}, and also from more general results of Connell and Muchnik~\cite{cm2007}.

The speed of the random walk is closely related to the Hausdorff dimension of the stationary measure. Work in~\cite{t2019} shows that:
\[\dim(\nu_\rho) = \frac{h}{\ell(\rho)},\]
where $h$ denotes the entropy of the random walk. 

Therefore Theorem~\ref{maintheorem} immediately translates into a statement about the behavior of the Hausdorff dimension of the stationary measure:

\begin{corollary}[Dimension drop of stationary measures]
For each \(\epsilon > 0\) there exists a compact \(K \subset \T(M)\)  such that \(\dim(\nu_{\rho}) < \epsilon\) for all \([\rho] \notin K\).
\end{corollary}

\begin{proof}
 This follows immediately from Theorem~\ref{maintheorem} and the above formula for \(\dim(\nu_\rho)\).
\end{proof}

\setcounter{maintheorem}{0}

\section{Preliminaries}\label{preliminaries}

\subsection{Teichm\"uller Geometry}

We now recall the definitions of the Teichmüller space and Teichmüller distance.  General references on the subject are~\cite{teichbook1} and~\cite{teichbook2}.

Let $M=M_{g,n}$ be a compact surface of genus $g$ with $n$ boundary components. A \emph{complex structure} on $M$ is an atlas of charts $z_{\alpha}: U_{\alpha}\to\mathbb{C}$ where at the intersection $z_{\alpha}(U_{\alpha}\cap U_{\beta})$ the transition maps $z_{\beta}\circ z_{\alpha}^{-1}$ are biholomorphisms. 

The Teichm\"uller space $\mathcal{T}(M)$ of $M$ is defined as the equivalence classes of complex 
structures on $M$, where two complex structures $X$ and $Y$ are \emph{equivalent} 
if there is a biholomorphism $f:X\to Y$ which is isotopic to the identity. 
Equivalently, $\mathcal{T}(M)$ can be thought of as the space of equivalence 
classes of hyperbolic structures on $M$, where two hyperbolic structures $X$ and $Y$ are equivalent if there is an isometry $f:X\to Y$ that is isotopic to identity. Yet another definition is given by considering the representations of the fundamental group $\Gamma=\pi_1(M)$. Let $G=\isom^{+}(\mathbb{H})$ be the orientation preserving isometries of the upper half-plane. The Teichm\"uller space of $M$ is also the space of
discrete faithful representations from $\Gamma$ into $G$ considered up
to conjugation by elements of $G$. We denote the equivalence class of a representation $\rho: \Gamma\to G$ with $[\rho]$. We endow \(\T(M)\) with the subspace topology inherited from the space of
representations from \(\Gamma\) to \(G\).

In order to define a metric on the Teichm\"uller space, we switch back to the complex analytic definition and consider a map $f:X\to Y$ where $X$ and $Y$ are two complex structures on $M$. Let us define
\[
K_{f}(p) \coloneqq \dfrac{\abs{f_z(p)}+\abs{f_{\bar{z}}(p)}}{\abs{f_z(p)}-\abs{f_{\bar{z}}(p)}}
\]
the \emph{quasi-conformal dilatation} of $f$ at $p\in X$.  This quantity is well defined as the coordinate charts are conformal. 

The \emph{quasi-conformal dilatation} of $f$ is defined as 

\[
K_f= \underset{p \in X}{\sup} \{K_{p}(f)\}.
\]

The map $f$ is called \emph{quasiconformal} if $K_f<\infty$. One can see that $K_f\ge 1$, and $K_f=1$ if and only if $f$ is conformal. The \emph{Teichm\"uller distance} $\dist_{Teich}$ on $\mathcal{T}(M)$ is defined as follows:

\[
\dist_{Teich}(X,Y)=\frac{1}{2} \inf_{f\simeq id} \{ \log (K_f) \mid f:X\to Y\},
\]
where $f$ is any quasi-conformal map isotopic to the identity.

Let $\gamma$ be a non-trivial essential simple closed curve on $M$, and $(X, \dd z)$ be a complex structure on $M$. The \emph{extremal length} of $\gamma$ on $X$ is defined as

\[
\ext_{X}(\gamma)=\sup \frac{L^{2}_{\sigma}(\gamma)}{A(\sigma)}
\]
where the supremum is over all conformal metrics $\sigma(z)\abs{\dd z}$,  and 
\[
L_{\sigma}(\gamma)=\inf _{\gamma\sim\gamma'}\int_{\gamma'} \sigma(z)\abs{\dd z}
\text{\ ,\ } 
A(\sigma)=\int_{X}\sigma^{2}(z)\abs{\dd z}^2. 
\]

\subsection{Random walks on hyperbolic spaces}
\label{sec:converging}

In this subsection we provide the context for Theorem~\ref{thm:converging_actions}. Background on random walks on groups is available in~\cite{randombook1},~\cite{randombook2}, and~\cite{randombook3}.   A general reference on hyperbolic spaces is~\cite{hyperbolicbook}. 

Let $(X, \dist)$ be a metric space. The Gromov product between points \(x,y\) in $X$ with respect to a third point \(o\) is defined as
\[(x,y)_o = \frac{1}{2}\left(\dist(o,x)+\dist(o,y)-\dist(x,y)\right).\]

The space $X$ is called \(\delta\)-\emph{hyperbolic} (or \emph{hyperbolic} for short) for some \(\delta \ge 0\) if for all \(w,x,y,z\in X\), it holds that
\[\min((x,y)_w, (y,z)_w)  \le (x,z)_w + \delta. \] 

A sequence of pointed metric spaces \((X_k,o_k)\) is called \emph{uniformly hyperbolic} if there exists \(\delta \ge 0\) such that \(X_k\) is \(\delta\)-hyperbolic for all \(k\).

Let $X$ be a hyperbolic space, and let $\Gamma$ be a finitely generated group which acts on $X$ by isometries. Let $\mu$ be a probability measure on $\Gamma$ which is \emph{admissible}, that is, the semi group generated by the support of $\mu$ is equal to $\Gamma$. The measure $\mu$ determines a random walk on $\Gamma$ by taking 
\[
Z_n=g_1g_2\dotsm g_n
\]
where $g_i$ are i.i.d elements of $\Gamma$ with common distribution $\mu$. Fixing a base point $x\in X$ defines a sequence of points $Z_n x=g_1g_2\dotsm g_n x$ in $X$, and it is called a random walk in $X$ \emph{driven by} $\mu$. 

We say that $\mu$ has a \emph{finite first moment} if
$\sum_{g \in \Gamma} \abs{g}\mu(g) < \infty$, where $\abs{\cdot}$ is some
word metric on $\Gamma$ with respect to a finite generating set.

We now consider the case where \(X=(\H,\dist)\) is the hyperbolic plane, with the distance function induced from the hyperbolic metric $ds^2=\dfrac{dx^2+dy^2}{y^2}$. Fix a base point \(o\in \H\), and let
\(G=\isom^{+}(\H)\) be the group of orientation preserving isometries of \(\H\).  Given a
discrete faithful representation \(\rho:\Gamma \to G\)  we consider the 
\emph{speed (or linear drift)}
of the induced random walk \(\rho(Z_n)o\) on \(\H\), which is the number
\(\ell(\rho)\) (given by Kingman's subadditive ergodic theorem) such that almost
surely
\[\ell(\rho) \coloneqq \lim\limits_{n \to +\infty}\frac{1}{n}\dist(o,\rho(Z_n)o).\]

This definition is equivalent to the one given in the introduction (see e.g.\ \cite{FM}). We claim that \(\ell\) is well defined as a function on \(\T(M)\): suppose that \(\widetilde{\rho} = g^{-1}\rho g\) and notice that
\[\dist(o,\widetilde{\rho}(Z_n)o) = \dist(o,g^{-1}\rho(Z_n)go) = \dist(go,\rho(Z_n)go).\]
The claim now follows from the observation that the right-hand side of the above equation differs from \(\dist(o,\rho(Z_n)o)\) in absolute value by at most \(2\dist(o,go)\).

We now describe the set-up for \Cref{thm:converging_actions} from the introduction. 




\begin{definition}
\label{def:converging_action}
Consider a group $\Gamma$, and a sequence of pointed metric spaces $(X_k, o_k)_{k \in \N\cup \{\infty\}}$, each of them endowed
with an isometric action $\rho_k$ of $\Gamma$. We say that this sequence of actions converges if,
for each $g\in \Gamma$, the distance $\dist(o_k, \rho_k(g)o_k)$ converges to
$\dist(o_\infty, \rho_\infty(g)o_\infty)$ as $k\to \infty$.
\end{definition}

We say that a
probability measure $\mu$ on $\Gamma$ is \emph{non-elementary} for an isometric
action \(\rho\) on a hyperbolic space if there exist two elements in the semigroup generated by the
support of $\mu$ which act through $\rho$ as two independent loxodromic
isometries, i.e., their sets of fixed points at infinity are disjoint.

Let \(\Gamma\) be a countable group, \(\mu_k\) be
a sequence of probability measures on \(\Gamma\) which converges pointwise to a
probability measure \(\mu_\infty\).  Let $(X_k, o_k)_{k\in \N\cup \{\infty\}}$ be
a uniformly hyperbolic sequence of pointed metric spaces, and $\rho_k:\Gamma
\to \isom(X_k), k\in\N$ be sequence of isometric actions of \(\Gamma\)
that converges to an action $\rho_\infty$. Assume that $\mu_\infty$ is non-elementary for $\rho_\infty$.

Consider for each $k$ a random walk \(Z_n^{(k)} = g_1^{(k)}\dotsm g_n^{(k)}\) where $g_i^{(k)}$ are i.i.d.\ elements with common distribution \(\mu_k\). For each \(k \in \N \cup \lbrace \infty\rbrace\), we define \(\ell(\mu_k) \in [0,+\infty]\) as the almost sure limit
\[\ell(\mu_k) = \lim\limits_{n \to +\infty}\frac{1}{n}\dist(o_k,\rho_k(Z_n^{(k)})o_k).\]
\Cref{thm:converging_actions} states that for such a sequence of actions,  
$\liminf \ell(\mu_k) \geq \ell(\mu_\infty)$. We will deduce Theorem~\ref{thm:converging_actions} from the more precise Proposition~\ref{prop:large_deviations},
that gives uniform exponential large deviation estimates along the sequence $\mu_k$.

\section{Proofs of Theorems~\ref{maintheorem} and~\ref{teichdistancebound}}

We will now explain how we can deduce Theorems~\ref{maintheorem} and~\ref{teichdistancebound} 
from the semicontinuity statement of Theorem~\ref{thm:converging_actions}.

\subsection{Proof of Theorem~\ref{maintheorem}}
Fix a finite symmetric generating set \(F \subset \Gamma\) containing the identity.  By~\cite[Proposition 2.1]{b1988} there exists for each \(\rho\) a basepoint \(o_{\rho} \in \H\) such that
\[\max\limits_{\gamma \in F}\dist(o_{\rho},\rho(\gamma)o_\rho)  = \min\limits_{x \in \H}\max\limits_{\gamma \in F}\dist(x,\rho(\gamma)x).\]
We define the rescaling factor \(R_\rho = R_{\rho, F}\) as the common value of both sides of the equation above.
From definition, it follows that \(R_\rho\) is continuous and proper on \(\mathcal{T}(M)\).
Consider the rescaled distance \(\dist_\rho = R_\rho^{-1} \dist\) on $\mathbb{H}$. 

Recall that an \(\R\)-tree is a non-empty metric space which is \(0\)-hyperbolic and such that every pair of points is joined by a unique geodesic.   
The following is the main result of~\cite{b1988} and~\cite{p1988} 
(following previous work~\cite{cm1987} and~\cite{ms1984}).

\begin{lemma}\label{rtreelimit}
Each sequence \(\rho_n\) such that \([\rho_n]\) leaves every compact subset in \(\T(M)\), has a subsequence \(\rho_{n_k} \) with \(n_k \to +\infty\) such that \(\rho_{n_k}\) when viewed as an action on \(\H\) endowed with the distance \(\dist_{\rho_{n_k}}\) and the
basepoint \(o_{\rho_{n_k}}\), converges to an action \(\rho_T\) on an \(\R\)-tree \((T,\dist_T,o_T)\).

Furthermore, the group \(\rho_T(\Gamma)\) acts minimally on \(T\) (i.e., there is no proper closed invariant subtree), and for any arc \(I\) in \(T\) the set of \(\gamma \in \Gamma\) such that \(\rho_T\) stabilizes \(I\) is a virtually abelian subgroup of \(\Gamma\).
\end{lemma}

We now verify that the action \(\rho_T\) is non-elementary.

\begin{lemma}[Non-elementary action on the \(\R\)-tree]\label{nonelementarytreelemma}
Let \(\rho_T\) be a representation of \(\Gamma\) into the isometry group of an \(\R\)-tree \((T,\dist_T)\) with the property that stabilizers of arcs are virtually abelian.

Then there exist \(\gamma_1,\gamma_2 \in \Gamma\) such that \(\rho_T(\gamma_1)\) and \(\rho_T(\gamma_2)\) are loxodromic isometries of \(T\) along geodesics whose intersection is either empty or a compact arc.
\end{lemma}
\begin{proof}
The action of \(\rho_T(\Gamma)\) is irreducible in the sense that there is no global fixed point on the boundary at infinity (see~\cite[Proposition 2.6]{p1989}).   
The existence of the two required loxodromic elements now follows from~\cite[Proposition 3.7]{chiswell}.
\end{proof}

In view of the above lemmas, Theorem~\ref{maintheorem} follows immediately from Theorem~\ref{thm:converging_actions} 
and in fact we obtain a lower bound in terms of the rescaling factor \(R_\rho\).

\begin{theorem}\label{scalingconstanttheorem}
 There exists a constant \(c > 0\) such that \(\ell([\rho]) \ge c R_{\rho}\) for all \([\rho] \in \T(M)\).
\end{theorem}
\begin{proof}
Suppose by contradiction that we may find a sequence of representations \(\rho_n\) for which
\(\ell([\rho_n]) / R_{\rho_n}\) tends to zero. Extracting a subsequence, we may assume 
that \(\rho_{n_k}\), viewed as an action on \(\H\) endowed with the distance \(\dist_{\rho_{n_k}}\), converges
to a non-elementary action, either on a tree if \(\rho_n\) escapes to infinity by Lemma~\ref{rtreelimit}, or on 
\(\H\) itself otherwise. 

Notice that \(\ell([\rho_n])/R_{\rho_n}\) is the speed of \(\rho(Z_n)o\) with respect to the distance \(\dist_\rho\).
We may apply Theorem~\ref{thm:converging_actions} to deduce that the liminf of this speed is bounded below by the speed 
in the limiting action. As the speed of a nonelementary action is always positive, we deduce that this liminf is positive,
a contradiction.
\end{proof}

\subsection{Proof of Theorem~\ref{teichdistancebound}}

In view of Theorem~\ref{scalingconstanttheorem}, to prove Theorem~\ref{teichdistancebound} 
it suffices to find a set of curves $F$ such that the rescaling factor  \(R_\rho=R_{\rho, F}\) can be bound from below by a 
multiple of \(\dist_{Teich}([\rho_0],[\rho])\).
We fix \(F \subset \Gamma\) to be a subset that is finite, symmetric and filling.

We denote by $\length_\rho(\gamma)$ the hyperbolic length of the geodesic representative of \(\gamma\), that is
\[\length_\rho(\gamma) = \min\limits_{x \in \H}\dist(x,\rho(\gamma)x).\]
We observe that
\begin{equation}
\label{eq:Rrho_ge}
  R_\rho = \min\limits_{x \in \H}\max\limits_{\gamma \in F}\dist(x,\rho(\gamma)x)
\ge \max_{\gamma\in F}\length_\rho(\gamma).    
\end{equation}
So to obtain Theorem~\ref{teichdistancebound}, we need a lower bound for the right-hand side above.

Let \(\ext_\rho(\gamma)\) denote the extremal length of the curve \(\gamma\) under the conformal structure provided by \(\rho\).
As proven by Maskit~\cite[Corollary 3]{m1985}, we have
\begin{equation}\label{eq:maskitlowerineqhyplength}
	\frac{1}{2}\length_\rho(\gamma)e^{\length_\rho(\gamma)/2}\ge \ext_\rho(\gamma),
\end{equation}
so it suffices to obtain a lower bound on extremal length. This will be obtained from a
result of Walsh~\cite[Lemma 3]{w2019}. 
While we do not need the specific details in Walsh, we include some of them for coherence.

For a unit area quadratic differential \(q\) based at 
some basepoint \([\rho_0]\), denote \(R(q;t)\) the point in Teichmüller space obtained 
after following a Teichmüller ray for time \(t>0\) in the direction provided by \(q\). 
Let \(V(q)\) (respectively \(H(q)\)) be the vertical (respectively horizontal) foliation of \(q\). 
The union of vertical saddle connections of \(V(q)\) is called its critical graph. The complement of the critical graph decomposes into finitely many components (the number bounded above in terms of the Euler characteristic) each of which is either a cylinder \(C\) or a minimal component \(V\) with every leaf dense. The transverse measure restricted to a minimal component \(V\) is a linear combination \(\sum_j m_{V, j}\) of distinct ergodic measures \(m_{V,j}\). 
Each pair \( V_j = (V, m_{V, j} )\) is said to be an indecomposable component of \(V(q)\).

Walsh proved the following inequality
\begin{equation}\label{eq:walshlowerineqextrlength}
	e^{-2t}\ext_{R(q;t)}(\gamma)\ge E^2_q(\gamma),
\end{equation}
where 
\[
E^2_q (\gamma) = \sum_{V_j} \frac{i(V_j, \gamma)^2}{i(V_j, H(q))}
\]
in which \(i(\ast, \ast)\)
denotes the geometric intersection number.

We will use the fact that \(F\) is filling to derive a uniform (over $q$) lower bound on \(\max_{\gamma \in F} E_q(\gamma)\).

\begin{lemma}\label{le:uniformlowerboundintersection}
	Given a basepoint \([\rho_0] \in \T(M)\) there is some \(c>0\) such that
	\[
		\inf_{q\in T^1([\rho_0])}\max_{\gamma\in F} E_q(\gamma) >c,
	\]
	where the infimum is taken over all unit area quadratic differentials at \([\rho_0]\).
\end{lemma}
\begin{proof}
For any $q \in T^1 ([\rho_0])$, we have $i(V_j, H(q)) \le i (V(q) , H(q)) = \textrm{Area}(q) = 1$.
This implies $E^2_q (\gamma) \ge \sum_{V_j} i(V_j, \gamma)^2$.

	Assume we have a sequence \(q_n\) of unit area quadratic differentials at \([\rho_0]\) 
	such that \(\max_{\gamma \in F} E_{q_n} (\gamma) \) converges to \(0\) for all $j$
 Since the space of unit area quadratic differentials at a basepoint is compact we can pass to a subsequence that converges to some \(q\). Furthermore, since geometric intersection number is continuous we have \( E_q(\gamma) =0\) for all \(\gamma\in F\).
 In particular, this implies $i(V(q), \gamma) = 0$ for all $\gamma \in F$. 
 This is impossible because $F$ is a filling set.
\end{proof}

We use the above lemma to get the following global lower bound on the maximal lengths over~\(F\).
\begin{lemma}\label{le:globalboundlengths}
	Given a basepoint \([\rho_0] \in \T(M)\) there are some \(c_1,c_2>0\) such that, for all \([\rho] \in T(M)\),
	\[
		\max_{\gamma\in F}\ext_{\rho}(\gamma)\ge c_1 e^{2\dist_{Teich}([\rho_0],[\rho])}
	\]
	and hence
	\[
		\max_{\gamma\in F}\length_{\rho}(\gamma)\ge c_2 \dist_{Teich} ([\rho_0],[\rho])
	\]
	where \(\dist_{Teich}\) denotes the Teichmüller distance.	
\end{lemma}
\begin{proof}
	Let \(q\) be such that  \([\rho] = R(q;\dist_{Teich}([\rho_0],[\rho]))\). By \cref{eq:walshlowerineqextrlength} we have
	\[
		\max_{\gamma\in F}\ext_{\rho}(\gamma)\ge e^{2\dist([\rho_0],[\rho])}\max_{\gamma\in F} E^2_q(\gamma),
	\]
	and so by \cref{le:uniformlowerboundintersection} we get the first inequality. By \cref{eq:maskitlowerineqhyplength} we have
	\[
		\max_{\gamma\in F} \frac{1}{2} \length_{\rho}(\gamma)+\log \max_{\gamma\in F} \frac{1}{2} \length_{\rho}(\gamma) \ge 2\dist([\rho_0],[\rho])+\log( c_1),
	\]
	so the second inequality in the lemma is asymptotically satisfied for any \(c_2\) slightly smaller than 2. Furthermore, given any bounded domain we can choose \(c_2\) small enough so the inequality is satisfied.
\end{proof}

Together with~\eqref{eq:Rrho_ge}, Lemma~\ref{le:globalboundlengths} gives a lower bound \(R_\rho \geq c_2 \dist_{Teich} ([\rho_0],[\rho])\). With Theorem~\ref{scalingconstanttheorem}, this concludes the proof of Theorem~\ref{teichdistancebound}.

\section{Proof of Theorem~\ref{thm:converging_actions}}

In  this section, we prove Theorem~\ref{thm:converging_actions}. To have lighter notation, we will keep the action implicit and write
$go_k$ instead of $\rho_k(g) o_k$. Since there is only one possible action for each basepoint
$o_k$, this should not create confusion.

\subsection{Schottky sets}

Following~\cite{g2021},  a finite set of isometries \(S\) acting on a Gromov hyperbolic space $X$, is said to be \((\eta,C,D)\)-Schottky if the following three conditions are satisfied:
\begin{enumerate}
 \item For all \(x,y\in X\) the proportion of \(a \in S\) such that \((x,ay)_o \le C\) is at least \(1-\eta\).
 \item For all \(x,y\in X\) the proportion of \(a \in S\) such that \((x,a^{-1}y)_o \le C\) is at least \(1-\eta\).
 \item For all \(a \in S\) one has \(\dist(o,ao) \ge D\).
\end{enumerate}

The following is a sufficient condition for a set \(S\) to be Schottky with certain parameters, which depends on checking conditions involving only a finite number of points.
Since the notion of convergence we use in Definition~\ref{def:converging_action} only
gives controls for finitely many points at a time, this criterion will enable us
to construct finite sets which are Schottky sets uniformly along a converging family of representations.

\begin{lemma}[Schottky set criterion]\label{schottkycriteria}
Let \((X,\dist)\) be a \(\delta\)-hyperbolic metric space with a basepoint \(o \in X\).
Suppose \(S\) is a finite symmetric set of isometries of \(X\) such that \(c_1 + 2\delta < c_2/2\) where
\(c_1 = \max\limits_{g \neq  h, g,h \in S} (go,ho)_o,\) and \(c_2 = \min\limits_{g \in S}\dist(o,go)\).

Then \(S\) is an \((\eta,C,D)\)-Schottky set with \(\eta = \frac{2}{\#S}\) and \(C = c_1 + 3\delta\) and \(D = c_2\).
\end{lemma}
\begin{proof}
 Let \(\epsilon \in (2\delta, c_2/2 - c_1)\) and for each \(g \in S\) set
 \[V(g) = \lbrace x \in X:  (x,go)_o \ge c_1 + \epsilon\rbrace.\]

\emph{Claim 1: If \(g \neq h\) then \(V(g) \cap V(h) = \emptyset\).}

Indeed if \(x \in V(g) \cap V(h)\) then one would have
\[c_1+\epsilon \le \min\{ (x,go)_o
, (x,ho)_o \} \le (go,ho)_o + \delta \le c_1 + \delta,\]
contradicting the fact that \(\delta < \epsilon\).

\emph{Claim 2: If \(x \notin V(g^{-1})\) then \(gx \in V(g)\).}

To see this observe that from the first condition one has
\[\frac{\dist(o,x) + \dist(o,g^{-1}o) - \dist(x,g^{-1}o)}{2} < c_1 + \epsilon,\]
while if \(gx \notin V(g)\) we would have
\[\frac{\dist(o,gx) + \dist(o,go) - \dist(gx,go)}{2} < c_1 + \epsilon.\]

Taking the sum this would imply
\[c_2 \le \dist(o,go) < 2c_1 + 2\epsilon,\]
contradicting the fact that \(\epsilon < \frac{1}{2}c_2 - c_1\).

\emph{Claim 3:  If \(x \in V(g)\) and \(y \in V(h)\) for \(g \neq h\) then \((x,y)_o \le c_1 + 2\delta\).}

By hyperbolicity one has
\[\min((x,go)_o,(x,ho)_o) \le (go,ho)_o + \delta \le c_1 + \delta.\]

Since \((x,go)_o \ge c_1 + \epsilon > c_1 + \delta\) this implies that \((x,ho)_o \le c_1 + \delta\).  From this we obtain
\[\min((y,ho)_o,(x,y)_o) \le (x,ho)_o + \delta \le c_1 + 2\delta,\]
but since \((y,ho)_o \ge c_1 + \epsilon > c_1 + 2\delta\) this implies \((x,y)_o \le c_1 + 2\delta\) as claimed.

\emph{Claim 4:  \(S\) is \((\eta,C,D)\)-Schottky for the constants in the statement.}

Let us check the first property in the definition of Schottky sets, as the second one follows by symmetry
of $S$ and the third one comes from the definition of $c_2$.
Given \(x,y \in X\) let \(a_1,a_2 \in S\) be distinct and such that \(x \notin V(a_i^{-1})\) for \(i = 1,2\).
By Claim 1 the \(a_i\) are chosen among at least \(\# S - 1\) elements of \(S\).
By Claim 2 one has \(a_ix \in V(a_i)\) for \(i=1,2\). By hyperbolicity and Claim 3 one has
\[\min((a_1x,y)_o,(a_2x,y)_o) \le (a_1x,a_2x)_o + \delta \le c_1 + 3\delta  = C.\]

This implies that either \((a_1x,y)_o \le C\) or \((a_2x,y)_o \le C\).
Hence the subset of \(S\) consisting of elements with \((ax,y)_o > C\) can have at most two elements.
\end{proof}

\subsection{Proof of Theorem~\ref{thm:converging_actions}}

In this paragraph, we prove Theorem~\ref{thm:converging_actions}.
Let us fix a sequence of pointed $\delta$-hyperbolic spaces $(X_k, o_k)$ endowed with
actions of a group $\Gamma$, and assume that $\rho_k$ converges to $\rho_\infty$ in
the sense of Definition~\ref{def:converging_action}. Let also $(\mu_k)$ be probability
measures on $\Gamma$ such that $\mu_k$ converges pointwise to $\mu_\infty$ and
the action of $\mu_\infty$ through $\rho_\infty$ is non-elementary on $X_\infty$.

The following lemma is a classical application of a ping-pong argument.
\begin{lemma}
\label{lem:Schottky_infty}
Let $\eta>0$. Then there exists $C>0$ such that, for any $D>0$, there exist $N$ and a finite symmetric
set $S$ in $\Gamma$ in the support of $\mu_\infty^N$ such that $\# S \ge 2/\eta$ and
\begin{equation}
\label{eq:construct_Schottky}
  \max\limits_{g \neq  h, g,h \in S} (g o_\infty, h o_\infty)_{o_\infty} < C - 3\delta,
  \quad \min\limits_{g \in S}\dist(o_\infty,g o_\infty) > D.
\end{equation}
\end{lemma}
\begin{proof}
This follows readily from the proof techniques of~\cite[Proposition~A.2]{boulanger_mathieu_sert_sisto}
or~\cite[Proposition~3.12]{g2021}.
\end{proof}

Let $\eta>0$. For suitable $C$ and $D$, we can consider a set $S$ as in Lemma~\ref{lem:Schottky_infty}.
By definition of converging actions, for large $n$ the inequalities
in~\eqref{eq:construct_Schottky} also hold for $\rho_k$. By Lemma~\ref{schottkycriteria}, it follows that $\rho_k(S)$ is
an $(\eta, C, D)$-Schottky set, uniformly for all large enough $k$. We can then use this Schottky set as
in~\cite{g2021}, to obtain quantitative estimates that are uniform in $k$. As a first example,
let us get a uniform version of~\cite[Lemma~4.14]{g2021}.

\begin{lemma}
\label{lem:dist_increase}
Let $\epsilon>0$. There exists $E$ such that, for all large $k$, for any $g\in \Gamma$,
\begin{equation*}
  \Pbb(\forall n,\ \dist(o_k, g Z_n^{(k)} o_k) \ge \dist (o_k, g o_k) - E) \ge 1 -\epsilon.
\end{equation*}
\end{lemma}
The intuition behind this lemma is that, given a Schottky set, then jumps from this
Schottky set will most of the time go towards infinity, yielding linear progress
from any starting point (and in particular small probability to go back towards the origin).
This is proved in~\cite{g2021} using the notion of pivotal
times. There is a pedagogical difficulty here: it would not make sense to repeat
exactly all the pivotal times computations of~\cite{g2021}, but we can not expect
our readers to be very familiar with this article. As a middle ground, we have decided to
extract a black box from~\cite{g2021}, in the form of the following lemma:
\begin{lemma}
\label{lem:dist_increase_aux}
Let $\delta>0$, $C>0$, $N>0$, $\alpha>0$, $\epsilon>0$. Then there exists $n_0 = n_0(\delta, C, N, \alpha, \epsilon)$
with the following property.
Consider a probability measure $\mu$ on the group $G$ of isometries
of a $\delta$-hyperbolic space $X$, and a set $S \subseteq G$ which
is $(1/100, C, 20C+100\delta+1)$-Schottky. Assume that $\mu^N$ gives mass at least $\alpha$ to
each element of $S$. Then, for any isometry $g \in G$, for any $o\in X$, there
exists a set $U$ of probability at least $1-\epsilon$ in $(\Omega, \Pbb) = (G^\N, \mu^{\otimes \N})$
such that, for each $\omega \in U$, there exists $j \leq n_0$ with
\begin{equation*}
  \forall n \ge n_0, \ \dist(o, gZ_n o) \ge \dist(o, g Z_j o)-2C-6\delta,
\end{equation*}
where $Z_n$ is the position of the random walk at time $n$.
\end{lemma}
This technical lemma is unfortunately not stated explicitly in~\cite{g2021}, but
it is proved there, as the first step of the proof of~\cite[Lemma~4.14]{g2021}.
The reader may either go and check in~\cite{g2021} that the $n_0$ given there indeed only depends
on $\delta, C, N, \alpha, \epsilon$, or trust us and accept this lemma as a black box.

\begin{proof}[Proof of Lemma~\ref{lem:dist_increase}]
Let $\eta=1/100$. Let $C$ be given by Lemma~\ref{lem:Schottky_infty} for this value of $\eta$.
Take $D=20C+100\delta+1$. By Lemma~\ref{lem:Schottky_infty} and the discussion that follows
it, we obtain a symmetric set $S \subseteq \Gamma$ in the support of $\mu_\infty^N$ for some $N$,
such that $\rho_k(S)$ is $(\eta, C, D)$ Schottky for all large $k$. For large enough $k$,
all the measures $\mu_k^N$ give a weight bounded away from zero to all
elements of $S$, say bounded from below by some $\alpha>0$, as $\mu_k$ converges pointwise to $\mu_\infty$.

Let $n_0 = n_0(\delta, C, N, \alpha,\epsilon/2)$ be given by Lemma~\ref{lem:dist_increase_aux}.
Then this lemma applies uniformly to all measures $\mu_k$ for large $k$:
for all $g\in \Gamma$, there exists a set $U=U(k, g)$ of probability at least $1-\epsilon/2$,
and some $j = j(k, g, \omega) \le n_0$ such that, on $U$,
\begin{equation}
\label{eq:qmlksdjvmlkjqsd}
  \forall n \ge n_0, \ \dist(o_k,gZ^{(k)}_n o_k) \ge \dist(o_k, g Z^{(k)}_j o_k)-2C-6\delta.
\end{equation}
It remains to control the $n_0$ first steps. Let $F$ be a finite subset of $\Gamma$ such that, with probability $>1-\epsilon/2$,
for all $i\le n_0$,
then $Z_i^{(\infty)}$ belongs to $F$. This property also holds for
large enough $k$, by pointwise convergence of $\mu_k$ to $\mu_\infty$.
There are finitely many points $(g o_\infty)_{g\in F}$. By convergence of the actions,
all the distances $\dist(o_k, g o_k)_{g \in F}$ are uniformly bounded for large $k$, by a
constant $C'$.

We obtain a set $V=V(k)$ of probability at least $1-\epsilon/2$ on which
\begin{equation}
\label{eq:ineq_le_n0}
  \forall i \le n_0, \ \dist(o_k, Z_i^{(k)}o_k) \le C'.
\end{equation}
The set $U(k,g)\cap V(k)$ has probability at least $1-\epsilon$. We claim that, on this set,
we have for all $n$ the inequality
\begin{equation*}
  \dist(o_k, g Z_n^{(k)} o_k) \ge \dist (o_k, g o_k) - 2C-6\delta - C',
\end{equation*}
proving the lemma with $E=2C+6\delta+C'$. Let us check this claim. First, if $n \le n_0$, then
$g Z_n^{(k)} o_k$ is within distance $C'$ of $g o_k$ by~\eqref{eq:ineq_le_n0},
and the result is obvious. Then,
for $n\ge n_0$, the result follows from the inequality~\eqref{eq:qmlksdjvmlkjqsd} together
with the fact that $g Z^{(k)}_j o_k$ is within distance $C'$ of $g o_k$ as $j \le n_0$.
\end{proof}

Let us now proceed to the lower estimate of the drift. As above, we extract a black
box result from~\cite{g2021}, specifying which properties of the measures are used.
\begin{lemma}
\label{lem:quantitative_bound}
Let $\delta>0$, $\eta > 0$, $\alpha\in (0,1)$, $C > 0$, $N > 0$, $A > 0$,  with $\eta A \geq C$.
Consider also a nonnegative real random variable $Q$, and $r\ge 0$ with
\begin{equation}
\label{eq:r_le}
  r < (1-40\eta) \frac{\Ebb(Q)}{NA}-2\eta.
\end{equation}
There exist $n_0$ and $\kappa>0$ only depending
on these quantities, with
the following property.

Consider a probability measure $\mu$ on the group $G$ of isometries
of a $\delta$-hyperbolic space $X$, and a set $S \subseteq G$ which
is $(\eta, C, 20C+100\delta+1)$-Schottky. Assume that $\mu^{2N} \geq \alpha \mu_S^2$,
where $\mu_S$ is the uniform probability measure on $S$. Assume moreover that the random walk
$Z^{(\nu)}$ driven by the
probability measure
$\nu = (\mu^{2N} - \alpha \mu_S^2)/(1-\alpha)$ satisfies, for any $g \in G$, the estimate
\begin{equation*}
  \Pbb(\forall n \geq 0,\ \dist(o, g Z_n^{(\nu)} o) \geq \dist(o, g o) - \eta N A) \geq 1 - \eta.
\end{equation*}
Finally, assume that the length of the jumps of $\mu^{NA}$ are stochastically bounded below
by $Q$: for all $k$, we have
\begin{equation*}
  \sum_{g \st \dist (o, go) \geq k} \mu^{NA}(g) \geq \Pbb(Q \geq k).
\end{equation*}

Then, for all $n \geq n_0$,
\begin{equation*}
  \Pbb(\dist(o, Z_n o) \leq r n) \leq e^{-\kappa n},
\end{equation*}
where $Z_n$ is the position of the random walk at time $n$.
\end{lemma}
Again, this lemma is not stated exactly in this form in~\cite{g2021}, but it is
proved there at the end of Section~5.3. The strategy is to decompose the walk along
successive time intervals of length roughly $NA$ (for which the size of the jumps is bounded
below by $Q$) interspersed with Schottky jumps that put the former in general position. This
ensures that the progress towards infinity is bounded below by a sum of independent
random variables distributed like $Q$, up to controlled error terms. The precise condition
on $r$ that shows up at the end of~\cite[Section~5.3]{g2021} is
\begin{equation*}
  r+\eta < \pare*{(1-\eta) \frac{\Ebb(Q)}{NA} -\eta}(1-22\eta) (1-17\eta),
\end{equation*}
which follows from~\eqref{eq:r_le}.

Let us deduce from this result uniform large deviations estimates along
a converging sequence of actions. We return to the standing assumptions of
this paragraph, with a converging sequence of actions of $\Gamma$ on hyperbolic spaces $X_k$, and
a pointwise converging sequence of probability measures $\mu_k$ on $\Gamma$
such that the action of $\mu_\infty$ on $X_\infty$ is non-elementary.

\begin{proposition}
\label{prop:large_deviations}
Let $\ell(\mu_\infty)$ be the drift of the random walk on $X_\infty$. Let $a < \ell(\mu_\infty)$.
Then there exists $\kappa>0$ such that, for all large enough $k$, for all $n\in \N$,
\begin{equation*}
  \Pbb(\dist (o_k, Z_n^{(k)} o_k) \leq a n) \leq e^{-\kappa n}.
\end{equation*}
\end{proposition}
\begin{proof}
Let $b \in (a, \ell(\mu_\infty))$.
Let $\eta>0$ be small enough (how small will be prescribed at the end of the proof), with $b+\eta < \ell(\mu_\infty)$.
Let $C$ be given by Lemma~\ref{lem:Schottky_infty} for this value of $\eta$.
Take $D=20C+100\delta+1$. By Lemma~\ref{lem:Schottky_infty} and the discussion that follows
it, we obtain a symmetric set $S \subseteq \Gamma$ in the support of $\mu_\infty^N$ for some $N$,
such that $\rho_k(S)$ is $(\eta, C, D)$ Schottky for all large $k$. For large enough $k$,
all the measures $\mu_k^N$ give a weight bounded away from zero to all
elements of $S$, as $\mu_k$ converges pointwise to $\mu_\infty$. In particular, for some
$\alpha>0$, one has $\mu_k^{2N} \geq 2\alpha \mu_S^2$, where $\mu_S$ is the uniform
measure on $S$.

The probability measures $\nu_k \coloneqq (\mu_k^{2N} - \alpha \mu_S^2)/(1-\alpha)$ converge
pointwise to $\nu_\infty = (\mu_\infty^{2N} - \alpha \mu_S^2)/(1-\alpha)$. Moreover,
$\nu_\infty$ acts in a non-elementary way on $X_\infty$ through $\rho_\infty$, as it
satisfies $\nu_\infty \geq \alpha \mu_S^2 / (1-\alpha)$ and therefore gives nonzero weight
to independent loxodromic elements since $S$ is Schottky. Therefore, Lemma~\ref{lem:dist_increase}
applies to the sequence of measures $\nu_k$ for $\epsilon=\eta$, yielding some constant $E$. Let $A$ be large
enough that $\eta A \geq C$ and $\eta N A \geq E$.

By subadditivity, we have $NA \ell(\mu_\infty) \le \sum_g \mu_\infty^{NA}(g) \dist(o_\infty, g o_\infty)$.
In particular, as $b+\eta < \ell(\mu_\infty)$, we may find a finite subset $F \subseteq G$ such that
\begin{equation*}
  NA (b + \eta) < \sum_{g \in F} \mu_\infty^{NA}(g) \dist(o_\infty, g o_\infty).
\end{equation*}
Let $\epsilon>0$. Let $Q$ be the real distribution with an atom of mass $\mu_\infty^{NA}(g) - \epsilon$ at
$\dist (o_\infty, g o_\infty) - \epsilon$ for each $g \in F$, and the missing mass put at $0$. For small
enough $\epsilon$, this random variable has expectation $> NA (b + \eta)$. Moreover,
by convergence of $\mu_k$ to $\mu_\infty$ and $\rho_k$ to $\rho_\infty$, the distribution of the size of the jumps
of $\mu_k^{NA}$ through $\rho_k$ is bounded below by $Q$, for all large enough $k$.

We apply Lemma~\ref{lem:quantitative_bound} to these quantities $\delta, \eta, \alpha, C, N, A, Q$, with
\begin{equation*}
  r = (1-40\eta) (b + \eta) - 2\eta,
\end{equation*}
which is indeed $<(1-40\eta) \Ebb(Q)/NA - 2\eta$. This lemma provides us with $n_0$ and $\kappa>0$.
By construction, for all large enough $k$, the measures $\mu_k$ all satisfy the assumptions of
the lemma. It follows that, uniformly in $k$ large, we have for all $n \geq n_0$
\begin{equation*}
  \Pbb(\dist (o_k, Z_n^{(k)} o_k) \leq r n) \leq e^{-\kappa n}.
\end{equation*}
When $\eta \to 0$, then $r=r(\eta)$ tends to $b>a$. Therefore, we may choose $\eta$ with $r>a$. We get
for all large $k$ and all $n\geq n_0$ the estimate
\begin{equation}
\label{eq:lkwjmlvkjwvx}
  \Pbb(\dist (o_k, Z_n^{(k)} o_k) \leq a n) \leq e^{-\kappa n}.
\end{equation}
It remains to handle each $n\in [1, n_0)$. For each such $n$, there exists
$g$ with $\mu^n_\infty(g)>0$ and $\dist (o_\infty, g o_\infty) > a n$, as otherwise the drift
$\ell(\mu_\infty)$ would be $\le a$. These two inequalities still hold for large $k$. It follows that
$\Pbb(\dist (o_k, Z_n^{(k)} o_k) \leq a n)$ is bounded away from $1$, uniformly for large $k$.
Decreasing $\kappa$ if necessary, we may therefore enforce~\eqref{eq:lkwjmlvkjwvx} separately
for each $n \in [1, n_0)$.
\end{proof}

\begin{proof}[Proof of Theorem~\ref{thm:converging_actions}]
Let $a < \ell(\mu_\infty)$. Proposition~\ref{prop:large_deviations} implies that,
for all large $k$, one has almost surely eventually $\dist(o_k, Z_n^{(k)} o_k) > a n$. As
$\ell(\mu_k)$ is the almost sure limit of $\dist(o_k, Z_n^{(k)} o_k) / n$, we get $\ell(\mu_k) \geq a$
for large $k$.
\end{proof}

\section{Singularity conjecture and open questions}



In this last section we return to the singularity conjecture and dimension drop of stationary measures. Recall from the introduction that by the results of~\cite{t2019}, the stationary measure \(\nu_{\rho}\) is exact dimensional and its dimension is given by
\begin{equation}\label{dimensionformula}\dim(\nu_\rho) = \frac{h}{\ell(\rho)},\end{equation}
where \(h = h(\mu)\) is the asymptotic (or Avez) entropy defined by
\[h = \lim\limits_{n \to +\infty}\frac{1}{n}H(Z_n),\]
and \(H(Z)=-\sum_{g\in \supp(Z)}\P(Z=g)\log(\P(Z=g))\) denotes the Shannon entropy of the random variable \(Z\). Note that \(h\) does not depend on the representation \(\rho\).

Recall the singularity conjecture from the introduction: 

\begin{conjecture}
If \(\mu\) is admissible and has finite support then there exists \(\delta < 1\) such that
\begin{equation*}
  \dim(\nu_\rho) \le \delta
\end{equation*}
for all \([\rho] \in \T(M)\).
\end{conjecture}

Since the visual boundary is one-dimensional, equation~\eqref{dimensionformula} implies that \(h \le \ell(\rho)\) for all \(\rho\). 
The singularity conjecture then amounts to this inequality being strict on all of \(\T(M)\).

Let us record some basic properties of \(\ell\).

\begin{proposition}
The function \(\ell:\T(M) \to (0,+\infty)\) is continuous.
\end{proposition}
\begin{proof}
This follows immediately from Theorem~\ref{thm:converging_actions}.  An alternative argument is via the 
Furstenberg formula~\cite[Theorem 18]{kl2011} for speed and convergence of the stationary measures.
\end{proof}

One basic result from \(\ell(p)\) being continuous and proper is as follows

\begin{corollary}
	The functions \(\ell:\T(M) \to [h,+\infty)\) and \(\dim(\nu):\T(M) \to (0,1]\) attain their minimum and maximum respectively.
\end{corollary}

It is natural to ask then the following question
\begin{question}
 Does \(\dim(\nu_\rho)\) attain its maximum at a unique point in \(\T(M)\)? Equivalently is \(\ell:\T(M) \to (0,+\infty)\) minimized at a unique point?
\end{question}

When the maximal dimension is 1 and \(\mu\) is symmetric (i.e., \(\mu(g) = \mu(g^{-1})\) for all \(g\)) the answer to the previous question is affirmative. That is, we have the following.

\begin{proposition}
	If \(\mu\) is admissible symmetric and has finite first moment, then there exists at most one point \([\rho] \in \T(M)\) such that \(\dim(\nu_\rho) = 1\).
\end{proposition}
\begin{proof}
	If \(\dim(\nu_\rho) = 1\) then, under the assumption that \(\mu\) is symmetric,  in fact \(\nu_\rho\) is absolutely continuous with respect to the visual measure (Lebesgue measure) on the boundary~\cite[Theorem 1.5]{bhm2011}.  Suppose \(\dim(\nu_{\rho_1}) = \dim(\nu_{\rho_2}) = 1\).
	
	There exists a quasi-conformal homeomorphism \(\varphi:\H \to \H\) such that \(\varphi(o) = o\) and \(\varphi\circ \rho_1(\gamma) = \rho_2(\gamma)\circ \varphi\) for all \(\gamma \in \Gamma\).   The quasi-conformal map \(\varphi\) extends continuously to the visual boundary in a unique way.  Denoting this extension by \(\varphi\) as well we have \(\varphi_*\nu_{\rho_1} = \nu_{\rho_2}\).
	
	This implies that the restriction of \(\varphi\) to the visual boundary is absolutely continuous.  However this can only happen if \([\rho_1] = [\rho_2]\) (see~\cite{a1985}).
\end{proof}

Another natural question to ask is the following one.

\begin{question}
 Is the function \(\ell:\T(M) \to (0,+\infty)\) (strictly) convex?
\end{question}

{\bf Acknowledgments:} This work is issued from an AIM Workshop in April 2022 on "Random walks beyond hyperbolic groups" organized by Joseph Maher, Yulan Qing and Giulio Tiozzo. We thank the organizers and the AIM for the opportunity to start this project: the working atmosphere of this workshop was extremely pleasant and productive.
We particularly thank Giulio Tiozzo for suggesting the question answered in this article. We thank all participants of this AIM workshop who discussed this question at some point, including Omer Angel, Adrian Carpenter, Vivian He, François Ledrappier, Joseph Maher, Catherine Pfaff and Giulio Tiozzo.

\bibliographystyle{alpha}
\bibliography{biblio}
\end{document}